\documentclass[1p]{elsarticle}
\usepackage{amsthm,amsmath,amssymb}
\usepackage{stmaryrd,mathrsfs}

\usepackage[backref=page]{hyperref}

\makeatletter
\def\ps@pprintTitle{%
 \let\@oddhead\@empty
 \let\@evenhead\@empty
 \def\@oddfoot{}%
 \let\@evenfoot\@oddfoot}
\makeatother

\usepackage{color}

\newcommand{\Dini}{\operatorname{Dini}}
\newcommand{\vertiii}[1]{{\left\vert\kern-0.25ex\left\vert\kern-0.25ex\left\vert #1
    \right\vert\kern-0.25ex\right\vert\kern-0.25ex\right\vert}}

\newcommand{\BMO}[0]{\operatorname{BMO}}
\newcommand{\supp}[0]{\operatorname{supp}}
\newcommand{\loc}[0]{\operatorname{loc}}

\renewcommand{\Re}[0]{\operatorname{Re}}

\newcommand{\R}{\mathbb{R}}

\newcommand{\Z}{\mathbb{Z}}

\newcommand{\Ss}{\mathscr{S}}

\newcommand{\D}{\mathscr{D}}

\theoremstyle{plain}
\newtheorem{theorem}[equation]{Theorem}

\newtheorem{corollary}[equation]{Corollary}
\newtheorem{lemma}[equation]{Lemma}

\theoremstyle{definition}

\numberwithin{equation}{section}

\begin{document}

\begin{frontmatter}

\title{$A_1$ theory of weights for rough homogeneous singular integrals and commutators}

\author[ehu,Ikerbasque,BCAM]{Carlos P\'erez\fnref{Fund1}}
\ead{carlos.perezmo@ehu.es}
\author[Is]{Israel P. Rivera-R\'{\i}os\fnref{Fund2}}
\ead{petnapet@gmail.com}
\author[UR]{Luz Roncal\fnref{Fund3}}
\ead{luz.roncal@unirioja.es}

\fntext[Fund1]{Supported by the Basque Government through the BERC 2014-2017 program and by Spanish Ministry of Economy and Competitiveness MINECO: BCAM Severo Ochoa excellence accreditation SEV-2013-0323 and the project MTM2014-53850-P.}
\fntext[Fund2]{Supported by Spanish Ministry of Economy and Competitiveness MINECO through the project MTM2012-30748.}
\fntext[Fund3]{Supported by Spanish Ministry of Economy and Competitiveness MINECO through the project MTM2015-65888-C04-4-P.}

\address[ehu]{Department of Mathematics, University of the Basque Country, Bilbao, Spain}
\address[Ikerbasque]{Ikerbasque, Basque Foundation of Science, Bilbao, Spain}
\address[BCAM]{BCAM - Basque Center for Applied Mathematics, Bilbao, Spain}
\address[Is]{IMUS and Departamento de An\'alisis Matem\'atico, Universidad de Sevilla, Sevilla, Spain}
\address[UR]{Departamento de Matem\'aticas y Computaci\'on, Universidad de La Rioja, Edificio CCT - C/Madre de Dios 53, 26006 Logro\~no, Spain}

\begin{abstract}
Quantitative $A_1-A_\infty$ estimates for rough homogeneous singular integrals $T_{\Omega}$ and commutators of $\BMO$ symbols and $T_{\Omega}$ are obtained. In particular the following estimates are proved:
\[
\|T_\Omega \|_{L^p(w)}\le c_{n,p}\|\Omega\|_{L^\infty}  [w]_{A_1}^{\frac{1}{p}}\,[w]_{A_{\infty}}^{1+\frac{1}{p'}}\|f\|_{L^p(w)}
\]
and
\[
\| [b,T_{\Omega}]f\| _{L^{p}(w)}\leq c_{n,p}\|b\|_{\BMO}\|\Omega\|_{L^{\infty}} [w]_{A_1}^{\frac{1}{p}}[w]_{A_{\infty}}^{2+\frac{1}{p'}}\|f\|_{L^{p}\left(w\right)},
\]
for $1<p<\infty$ and $1/p+1/p'=1$.
\end{abstract}



\end{frontmatter}




\section{Introduction and main results}

For $1< p < \infty$, the Muckenhoupt class $A_p$ is the set of weights $w$, that is, non-negative locally integrable functions, for which $w^{1-p'}\in L_{\loc}^{1}(\R^n)$ and
\begin{equation*}
[w]_{A_{p}}:=\sup_{Q}\left(\frac{1}{|Q|}\int_{Q}w\right)\left(\frac{1}{|Q|}\int_{Q}w^{-\frac{1}{p-1}}\right)^{p-1}<\infty,
\end{equation*}
where the supremum is taken over all cubes in $\R^n$. The quantity $[w]_{A_p}$ is called the $A_p$ constant of the weight and sometimes it is also called the weight characteristic.

The study of the sharp dependence on the $A_p$ constant of weighted inequalities for the main operators in Harmonic Analysis goes back to S. M. Buckley \cite{B}. For the Hardy--Littlewood maximal function $M$ he proved, for $1<p<\infty$, that
\[\|M\|_{L^p(w)\to L^p(w)}\leq c_{n,p} [w]_{A_p}^{\frac{1}{p-1}},
\]
where $w$ is an $A_p$ weight.
Later on, the so-called $A_2$ conjecture for the Ahlfors-Beurling transform, raised in \cite{AIS} and solved by S. Petermichl and A. Volberg in \cite{PV}, motivated the study of the sharp dependence on the $A_p$ constant of several operators such as the Hilbert transform, the Riesz transforms, or general Calder\'on-Zygmund operators (see for instance \cite{HA2,Le2,P1,P2} and the most recent developments in \cite{CAR,La,L, LN}), commutators (see \cite{CPP}), or the square function (see \cite{HL}).

Recently, the question of getting sharp quantitative weighted estimates has been also considered for another class of operators, namely the rough singular integrals. Let us consider $\Omega\in L^\infty(\mathbb{S}^{n-1})$ such that $\int_{\mathbb{S}^{n-1}}\Omega=0$. We define the kernel \[K(x)=\frac{\Omega(x')}{|x|^n}\] where $x'=\frac{x}{|x|}$. It is clear that $K$ is homogeneous of degree $-n$. We also observe that it satisfies the standard size estimate, but no angular smoothness is imposed. We define now \[T_\Omega f(x)= \text{p.v.} \int_{\mathbb{R}^n}\frac{\Omega(y')}{|y|^n}f(x-y)dy.
\]
We refer to \cite{Duo-HomenajeRubio} for a very nice review of the theory of these singular integrals. More recently, T. P. Hyt\"onen, the third author, and O. Tapiola obtained the following result in \cite{HRT} (there, the result is stated in terms of $A_p-A_{\infty}$ weights as introduced in \cite{HP1,HPR1}, but here we state it just in terms of the $A_p$ constant, for clarity)
\begin{equation}\label{HRTAp}
\|T_\Omega\|_{L^p(w)\to L^p(w)}\leq c_{n,p}[w]_{A_p}^{2\max\big\{1,\frac{1}{p-1}\big\}},
\end{equation}
which holds for every $w\in A_p$. This result can be seen as an updated version of the theorem of J. Duoandikoetxea and J. L. Rubio de Francia \cite{DuoRdF}. In \cite{HRT} it is conjectured that the dependence on the $A_2$ constant is linear. To prove or disprove such conjecture does not seem to be a trivial question.

As an application of this result, a commutator estimate with $\BMO$ functions can be derived by means of the conjugation method as introduced in \cite[p. 621]{CRW} (see also \cite{ABKP,CPP,PRR}). Indeed, recall that given a linear operator $T$ and $b\in \BMO$ the commutator of Coifman-Rochberg-Weiss $[b,T]$ is defined as
\[[b,T]f(x)=b(x)Tf(x)-T(bf)(x).\]
Then \eqref{HRTAp} combined with the method from \cite{CPP} yields the following estimate
\[
\|[b,T_\Omega]f\|_{L^p(w)\to L^p(w)} \leq c_{n,p} \|b\|_{\BMO} [w]_{A_p}^{3\max\big\{1,\frac{1}{p-1}\big\}}.
\]
In this case it is  also possible to obtain a result in terms of $A_p-A_\infty$ constant (cf. \cite{HP1}) but again we have stated the estimate just in terms of the $A_p$ constant for clarity.  At this point we would also like to stress the fact
that the conjugation method
relies upon the fact that $\exp(sb)\in A_p$ for every $1<p<\infty$, $b\in\BMO$ and $s\in\mathbb{R}$ small enough. However, this is not true  for $p=1$, namely in the $A_{1}$ case.  Recall that a weight $w$ belongs to the class $A_1$ if there is a finite constant $\kappa>0$ such that
$$ Mw(y) \le \kappa w(y) \qquad\text{a.e.},
$$
and the infimum of those constants $\kappa$ is called the $A_1$ constant
of $w$ and is denoted by $[w]_{A_{1}}$. Now we turn our attention to the dependence on the $A_1$ constant of a weight $w$. C. Fefferman and E. M. Stein proved the following inequality for the Hardy--Littlewood maximal function in \cite{FS}
$$
\|Mf\|_{L^{1,\infty}}\le c\int_{\R^n}|f|Mw\,dx.
$$
B. Muckenhoupt and R. L. Wheeden naturally conjectured that the maximal operator $M$ could be replaced by a Calder\'on--Zygmund operator on the left hand side of the above inequality. In particular, for the Hilbert transform $H$ the conjecture stated that there exists a constant $c>0$ such that for every weight $w$ the following inequality holds
\[\|Hf\|_{L^{1,\infty}(w)}\leq c \|f\|_{L^1(Mw)}.
\]
This conjecture was disproved by M. C. Reguera and C. Thiele in \cite{RT}. Nevertheless, the question led to a weaker conjecture, the so-called $A_1$ conjecture, which states that there exists a constant $c>0$ such that for any $w\in A_1$ we have
\[\|Hf\|_{L^{1,\infty}(w)}\leq c [w]_{A_1}\|f\|_{L^1(w)}.
\]
Motivated by the $A_1$ conjecture, A. K. Lerner, S. Ombrosi, and the first author proved in \cite{LOP1} the following estimate for any Calder\'on-Zygmund operator~$T$

\begin{equation}\label{A1LOP}\|Tf\|_{L^{1,\infty}(w)}\leq  c[w]_{A_1}\log\left(e+[w]_{A_1}\right) \|f\|_{L^1(w)}.
\end{equation}
To obtain \eqref{A1LOP}, the key point was to prove first the following special two-weight $L^p$ estimate with a good control in the bound: if \,$1<p<\infty$\,  and  $1<r<\infty$,
\begin{equation}\label{LpMrT}
 \|Tf\|_{L^p(w)} \leq c_T p p' (r')^{\frac{1}{p'}} \|f\|_{L^p(M_r w)}  \qquad w\geq 0,
\end{equation}
where as usual if $p>1$, $1/p+1/p'=1$ and $M_r$ is defined in \eqref{DefMq}. As a consequence, the following sharp control of $\|T\|_{L^p(w)\to L^p(w)}$ in terms of $p$ and of the $A_1$ constant follows:
\begin{equation}\label{A1pmay1}
\|T\|_{L^p(w)\to L^p(w)}\leq cpp'[w]_{A_1},
\end{equation}
for $1<p<\infty$. In the recent paper \cite{NRVV} the $A_1$ conjecture is disproved by using the Bellman function technique. In fact it is established there that the logarithmic term cannot be completely removed, so \eqref{A1LOP} seems to be sharp.

The aim of this paper is to obtain $A_1$ bounds in the spirit of \eqref{A1pmay1} for rough homogeneous singular integrals and also for their commutators with BMO functions.
The method we follow consists of obtaining two-weight estimates (Theorem \ref{thm:main1} below) and then derive mixed $A_1-A_{\infty}$ type results (Corollary \ref{cor1} below). The proof of those estimates are based on a decomposition of $T_\Omega$ as a sum of Calder\'on-Zygmund operators (see Section \ref{PrelimNotation}) and a suitable use of interpolation with change of measures, following the ideas of \cite{DuoRdF}, \cite{HRT}, and \cite{W}. For $q>0$, let us define $M_{q}$ by
\begin{equation}
\label{DefMq}
M_{q}f(x):=\sup_{Q\ni x} \Big(\frac{1}{|Q|}\int_Q|f|^{q}\,dt\Big)^{1/q}.
\end{equation}
Our main result is the following.

\begin{theorem}
\label{thm:main1}Let $T_{\Omega}$ be a \textit{rough homogeneous singular integral} and $b\in \BMO$. Let $1<p,r<\infty$. Then, for all $w\ge0$, $0<\theta<1$, we have
\begin{equation}\label{eq:desigr}
 \|T_\Omega f\|_{L^p(w)}\le c_{n,p}\|\Omega\|_{L^\infty}\frac{1}{1-\theta}(r')^{\frac{\theta}{p'}}\|f\|_{L^p(M_{r/\theta}(w))}
 \end{equation}
 and
\begin{equation}\label{eq:desigrComm}
\| [b,T_{\Omega}]f\| _{L^{p}(w)}\leq c_{n,p}\|b\|_{\BMO}\|\Omega\|_{L^{\infty}}\frac{1}{1-\theta}(r')^{\theta\big(1+\frac{1}{p'}\big)}\|f\|_{L^{p}(M_{r/\theta}(w))}.
 \end{equation}
\end{theorem}

Observe that the estimate \eqref{eq:desigr} is similar in spirit to \eqref{LpMrT} but worse due to the influence of the parameter $\theta$ which comes from the use the interpolation theorem with change of measure. Even though it is not clear how to avoid such interpolation theorem, we believe that our estimate is not optimal.
In fact, some modifications of the proof of the theorem could allow to state the estimate \eqref{eq:desigr} above in a more general form replacing
 the maximal function $M_{q}(w)$, $q\in (1,\infty)$,  defined using the $L^q$ norm, by a more general maximal function $M_{\Psi(L)}$
 as defined in \eqref{OrliczMax} below.  To be more precise, it is possible to prove an extension of \eqref{eq:desigr} as follows
 $$
 \|T_\Omega f\|_{L^p(u)}\le c_{n,p}\|\Omega\|_{L^\infty}\frac{1}{1-\theta}\|M_{\bar{\Psi}(L)}\|_{L^{p'}\rightarrow L^{p'}}^{\theta} \|f\|_{ L^p(M_{ \Psi_{\theta p}(L)   }(u)) }\qquad  0<\theta <1,
$$
where $\Psi_\rho(t)=\Psi\big(t^\frac{1}{\rho}\big)$ and $\bar{\Psi}$ is called the \textit{complementary function}  of $\Psi$ (see \eqref{complementaria}). The issue here is that the appearance of the parameter $\theta$ prevents obtaining a sharper result with logarithmic functions instead of power functions.  As a really sharp theorem we would expect results including just logarithmic type maximal function as those obtained in \cite{P} (cf. also \cite{HP2}) for the case of
classical Calder\'on-Zygmund operators. For instance, we conjecture that the following estimate holds for $p\in(1,\infty)$:
$$
 \|T_\Omega f\|_{L^p(w)}\le c_{n,p}\|\Omega\|_{L^\infty}\,  \|f\|_{L^p(M^{[p] +1}(w) )}   \qquad w\geq 0,
$$
which should be false for $M^{[p]}$, where $M^k$ is the $k$-th iterated maximal function, as in the case of Calder\'on-Zygmund operators.

With a suitable choice of $\theta$ in Theorem \ref{thm:main1} we will obtain mixed $A_1-A_\infty$ bounds similar to those obtained for the first time in \cite{HP1} (see also \cite{HP2}). Let us discuss first about the definition of the $A_{\infty}$ class and the $A_{\infty}$ constant. The $A_p$ classes are increasing with respect to $p$, so it is natural to define the $A_{\infty}$ class of weights as
$A_{\infty} = \cup_{p>1} A_p$. From the definition of $A_{\infty}$ it is not clear how to define an appropriate constant that characterizes the class. However, N. Fujii obtained essentially in \cite{F} the following characterization, rediscovered later by J. M. Wilson in \cite{Wi}: we will say that $w$ belongs to the $A_{\infty}$ class when the following quantity
\begin{equation}
\label{eq:infinity}
[w]_{A_{\infty}}:=\sup_{Q}\frac{1}{w(Q)}\int_{Q}M(\chi_{Q}w)dx
\end{equation}
is finite.
Here, $w(Q):=\int_Qw(x)\,dx$, and the supremum is taken over all cubes with edges parallel to the coordinates axes.

The following corollary summarizes the $A_1-A_\infty$ bounds, with the definition of $A_{\infty}$ in the sense just described, that we can obtain from Theorem~\ref{thm:main1}.

\begin{corollary} \label{cor1}
Let $1<p<\infty$. Then, if $w \in A_{\infty}$
$$
\|T_\Omega f\|_{L^p(w)}\le c_{n,p}\|\Omega\|_{L^\infty}  [w]_{A_{\infty}}^{1+\frac{1}{p'}}\, \|f\|_{L^p(Mw)}
$$
and
$$
\|[b,T_{\Omega}]f\| _{L^{p}(w)}\leq c_{n,p}\|b\|_{\BMO}\|\Omega\|_{L^{\infty}}[w]_{A_{\infty}}^{2+\frac{1}{p'}}\|f\|_{L^{p}\left(Mw\right)}.
$$
Therefore, if $w\in A_1$
$$
\|T_\Omega \|_{L^p(w)}\le c_{n,p}\|\Omega\|_{L^\infty}  [w]_{A_1}^{\frac{1}{p}}[w]_{A_{\infty}}^{1+\frac{1}{p'}}\,\|f\|_{L^p(w)}
$$
and
$$
\| [b,T_{\Omega}]f\| _{L^{p}(w)}\leq c_{n,p}\|b\|_{\BMO}\|\Omega\|_{L^{\infty}}[w]_{A_1}^{\frac{1}{p}}[w]_{A_{\infty}}^{2+\frac{1}{p'}}\|f\|_{L^{p}(w)}.
$$
\end{corollary}

Finally a direct application of \cite[Corollary 4.3]{Duo-JFA} leads us to the following result.

\begin{corollary} \label{cor2}
Let $1\le q<p<\infty$. Then we have
$$
\|T_\Omega \|_{L^p(w)}  \le c_{n,p,q}\|\Omega\|_{L^\infty}\,[w]_{A_q}^{2}
 \|f\|_{L^p(w)}
 $$
 and
 $$
\|[b,T_{\Omega}]f\| _{L^{p}(w)}\leq c_{n,p,q}\|b\|_{\BMO}\|\Omega\|_{L^{\infty}}[w]_{A_q}^{3}\|f\|_{L^{p}(w)}
$$
for all $w\in A_q$.
\end{corollary}

The rest of the paper is organized as follows. In Section \ref{PrelimNotation} we present basic definitions and properties and fix some notation. In Section \ref{Lemmata} we prove all the needed intermediate results. The proofs of Theorem~\ref{thm:main1} and Corollary \ref{cor1} are given in Section \ref{PruebasThmCor}.

\section{Preliminaries and notation}
\label{PrelimNotation}
The purpose of this section is to gather some definitions and basic properties that will be used throughout the rest of the paper.

\subsection{$\omega$-Calder\'on--Zygmund operators and commutators}

In Subsection \ref{sub:deco} it will be shown that $T_\Omega$ can be decomposed as a sum of $\omega$-Calder\'on--Zygmund operators with kernels that are \textit{Dini-continuous}. This decomposition is more precise than the one considered previously as in \cite{W}. For this purpose we recall in the present subsection the concept of $\omega$-Calder\'on--Zygmund operator. These are bounded linear operators on $L^2(\mathbb{R}^n)$ that admit the following representation
\[
Tf(x)=\int_{\mathbb{R}^n}K(x,y)f(y)\,dy
\]
for every smooth compactly supported function, provided that $x\not\in \supp f$, and the kernel $K$ satisfies the following properties
\begin{itemize}
\item (Size condition)
$$
|K(x,y)|\leq \frac{C_K}{|x-y|^n}.
$$
\item(Smoothness estimate) If $|x-y|>2|x-x'|$ then
\[
|K(x,y)-K(x',y)|+|K(y,x)-K(y,x')|\leq \omega\bigg(\frac{|x-x'|}{|x-y|}\bigg)\frac{1}{|x-y|^n}
\]
\end{itemize}
where $\omega:[0,\infty)\rightarrow[0,\infty)$ is a modulus of continuity, that is, an increasing subadditive ($\omega(s+t)\leq \omega(s)+\omega(t)$) function such that $\omega(0)=0$. Moreover, $\omega$ satisfies the \textit{Dini condition} if
\begin{equation}
\label{eq:Dini}
\|\omega\|_{\Dini}:=\int_0^1\omega(t)\frac{dt}{t}<\infty.
\end{equation}
In this case, the kernel $K$ is said to be a \textit{Dini-continuous kernel}.
For $\omega$-Calder\'on--Zygmund operators with $\omega$ satisfying the Dini condition, we will adopt the notation
\begin{equation}
\label{eq:CT}
C_T=\|T\|_{L^2\to L^2}+C_K+\|\omega\|_{\Dini}.
\end{equation}

In the last few years the study of sharp quantitative weighted inequalities for Calder\'on--Zygmund operators has led to a such a surprising result as the pointwise domination of $\omega$-Calder\'on--Zygmund operators by finite sums of sparse operators, see the definitions some lines below. First Lerner and F. Nazarov \cite{LN} and simultaneously J. M. Conde-Alonso and G. Rey \cite{CAR} obtained pointwise domination results in the case that $\omega$ satisfies the \textit{$\log$-Dini} condition, that is
\[
\int_0^1\omega(t)\log \frac{1}{t}\frac{dt}{t}<\infty.
\]
Recently M. T. Lacey \cite{La} relaxed that condition by allowing $\omega$ to satisfy the Dini-condition. After that, Hyt\"onen, the third author, and Tapiola \cite{HRT} obtained a fully quantitative result (see also \cite{L} for a short and elegant proof of the result).

\begin{theorem}[\cite{HRT,L}]
\label{thm:domination}
Let $T$ be an $\omega$-Calder\'on--Zygmund operator such that $\omega$ satisfies the Dini condition \eqref{eq:Dini}. Then for any compactly supported function $f\in L^1(\R^n)$ there exist $3^n$ dyadic lattices $\D_j$ and sparse collections $\mathscr{S}_j \subseteq \D_j$ such that
\begin{equation*}
  |Tf(x)| \ \le c_n C_T\sum_{j=1}^{3^n} \mathcal{A}_{ \mathscr{S}_j}f(x)
\end{equation*}
for almost every $x \in \R^n$, with $C_T$ as in \eqref{eq:CT}
and where
\[\mathcal{A}_\mathscr{S_j}f(x)=\sum_{Q\in\mathscr{S}_j}|f|_{Q}\chi_{Q}(x).
\]
\end{theorem}
In the theorem above and in what follows, the notation $f_Q$ stands for the average of a function $f$ over a cube $Q$:
$$
f_Q=\frac{1}{|Q|}\int_Qf(x)\,dx.
$$
We also recall that a $\eta$-sparse family $(0<\eta<1)$ is a family of cubes $\mathscr{S}$ contained in a dyadic lattice $\mathscr{D}$ such that for each $Q\in\mathscr{S}$ there exists a subset $E_Q\subseteq Q$ such that
\begin{enumerate}
\item $\eta|Q|\leq|E_Q|$;
\item the $E_Q$'s are pairwise disjoint.
\end{enumerate}
To learn about the notion of dyadic lattice and for a thorough study of sparse families we encourage the reader to consult \cite{LN}.

For commutators of Calder\'on--Zygmund operators whose modulus of continuity satisfies the Dini condition \eqref{eq:Dini}, the following sparse control has been very recently obtained in \cite{LORR}.

\begin{theorem}[\cite{LORR}]
\label{thm:ConmDomination}
Let $T$ be an $\omega$-Calder\'on--Zygmund operator such that $\omega$ satisfies the Dini condition \eqref{eq:Dini}, and let $b\in L_{\loc}^1(\R^n)$. Then for any compactly supported function $f\in L^\infty(\R^n)$ there exist $3^n$  dyadic lattices $\D_j$ and sparse collections $\mathscr{S}_j \subseteq \D_j$ such that
\begin{equation*}
  \big|[b,T]f(x)\big| \ \leq c_{n}C_{T}\sum_{j=1}^{3^{n}}\big(\mathcal{T}_{\mathscr{S}_{j},b}|f|(x)+\mathcal{T}_{\mathscr{S}_{j},b}^{*}|f|(x)\big)
\end{equation*}
for almost every $x \in \R^n$, with $C_T$ as in \eqref{eq:CT} and where
\[
\mathcal{T}_{\mathscr{S}_{j},b}|f|(x)=\sum_{Q\in\mathscr{S}_{j}}|b(x)-b_{Q}||f|_{Q}\chi_{Q}(x)
\]
and
\[
\mathcal{T}_{\mathscr{S}_{j},b}^{*}|f|(x)=\sum_{Q\in\mathscr{S}_{j}}\left(\frac{1}{|Q|}\int_{Q}|b(y)-b_{Q}|f(y)dy\right)\chi_{Q}(x).
\]
\end{theorem}

\subsection{A decomposition of homogeneous singular integrals in terms of $\omega$-Calder\'on--Zygmund operators}
\label{sub:deco}

Rough homogeneous singular integrals can be decomposed as a sum of $\omega$-Calder\'on--Zygmund operators. The idea of decomposing $T_\Omega$ as a sum of operators with some cancellation property goes back to \cite{Duo-Trans,DuoRdF,W} and has been refined recently in \cite{HRT}.  Indeed, following \cite[Section 3]{HRT}, we consider the following partition of unity. Let $\phi\in C_c^{\infty}(\R^d)$ be such that $\supp \phi\subset\{x:|x|\le \frac{1}{100}\}$  and $\int\phi\,dx=1$, and so that $\widehat{\phi}\in \mathcal{S}(\R^d)$. Let us also define $\psi$ by $\widehat{\psi}(\xi)=\widehat{\phi}(\xi)-\widehat{\phi}(2\xi)$. Then, with this choice of $\psi$, it follows that $\int\psi\,dx=0$. We write
$$\phi_j(x)=\frac{1}{2^{jn}}\phi\Big(\frac{x}{2^j}\Big) \quad \text{and}\quad  \psi_j(x)=\frac{1}{2^{jn}}\psi\Big(\frac{x}{2^j}\Big),
$$
and define the partial sum operators $S_j$ by $S_j(f)=f\ast \phi_j$.
Since $S_j f\to 0$ as $j\to-\infty$, for any sequence of integer numbers $\{N(j)\}_{j=0}^{\infty}$, with
$$
0=N(0)<N(1)<\cdots<N(j)\to\infty,
$$
we have the identity
$$
T_k=T_kS_k+\sum_{j=1}^{\infty}T_k(S_{k-N(j)}-S_{k-N(j-1)}).
$$
In this way, we can write
\begin{equation}\label{SumaTj}
T_{\Omega}=\sum_{j=0}^{\infty}\widetilde{T}_j^N,
\end{equation}
where
\begin{equation*}
\widetilde{T}_0^N:=\sum_{k\in\Z}T_kS_k,
\end{equation*}
and, for $j\ge1$,
\begin{equation*}
  \widetilde{T}_j^N :=\sum_{k\in\Z}T_k(S_{k-N(j)}-S_{k-N(j-1)}).
\end{equation*}
Each operator $\widetilde{T}_j^N$ is an $\omega$-Calder\'on--Zygmund operator with Dini-continuous kernel, as stated in the following lemma, (see \cite[Lemma 3.10]{HRT}, also for the notation).

\begin{lemma}[\cite{HRT}]
\label{lem:modulus}
The operator $\widetilde{T}_j^N$ is an $\omega_j^N$-Calder\'on--Zygmund operator with
\begin{equation*}
\omega_j^N(t)\le  c_n\|\Omega\|_{L^\infty} \min(1 , 2^{N(j)}t),
\end{equation*}
which satisfies
\begin{equation*}
\|\omega_j^N\|_{\Dini}:=   \int_0^1\omega_j^N(t)\frac{dt}{t}\leq c_n\|\Omega\|_{L^\infty}(1+N(j)).
\end{equation*}
\end{lemma}
Analogously, using \eqref{SumaTj}, $[b,T_\Omega]$  can be decomposed as a sum of commutators of $b$ and $\omega^N_j$-Calder\'on--Zygmund operators as follows
\begin{equation*}
[b,T_{\Omega}]f(x)=\sum_{j=0}^{\infty}[b,\widetilde{T}_{j}^{N}]f(x).
\end{equation*}

\subsection{The Reverse H\"older Inequality (RHI)}
\label{sub:Ap}
As mentioned in the introduction the $A_{\infty}$ constant  was studied thoroughly in \cite{HP1}  and the definition \eqref{eq:infinity} was proved to be the most suitable one.
One of the key results of that work was the following optimal reverse H\"older's inequality (see also \cite{HPR1}).

\begin{theorem}[Reverse H\"older inequality] \label{thm:SharpRHI}

Let \,$w \in A_{\infty}$, then there exists a dimensional constant $\tau_{n}$ such that
\begin{equation*}
  \left(\frac{1}{|Q|}\int_{Q}w^{r_{w}}\right)^{\frac{1}{r_{w}}}\leq\frac{2}{|Q|}\int_{Q}w.
\end{equation*}
where
$$r_w=1+\frac{1}{\tau_{n} [w]_{A_{\infty} }}.
$$
\end{theorem}
This sharp quantitative form of the RHI will be needed later.

\subsection{Orlicz maximal operators}

It can be easily checked that, for $1<r<\infty$, $Mf(x)\leq M_rf(x)$, where $M_r$ is the maximal operator defined in \eqref{DefMq}. Then one may wonder whether the standard Hardy--Littlewood operator could be generalized to a wider ``range of scales". For instance, whether it would be possible to build a maximal operator $\widetilde{M} $ such that
\begin{equation}\label{maximalScale}
cMf(x)\leq \widetilde M f(x) \leq c_r M_rf(x).
\end{equation}
Indeed, this can be done and for that purpose we need to build ``generalized local averages". That task can be adressed as follows. Let $\Psi$ be a Young function that is a continuous, convex, increasing function that satisfies $\Psi(0)=0$ and $\Psi(t)\to \infty$ as $t\to \infty$. Let $f$ be a measurable function defined on a set $E$ with finite measure of $\R^n$. We refer to \cite{PKJF} for more information. The \textit{$\Psi$-norm of $f$ over $E$} is defined by
$$
\|f\|_{\Psi(L),E}:=\inf\Big\{\lambda>0; \frac{1}{|E|}\int_{E}\Psi\Big(\frac{|f(x)|}{\lambda}\Big)\,dx\le 1\Big\}.
$$
If $\Psi$ is a Young function and $f$ is Lebesgue measurable, we define the \textit{Orlicz maximal operator} $M_{\Psi(L)}$ by
\begin{equation}\label{OrliczMax}
M_{\Psi(L)} f(x)=\sup_{Q\ni x}\|f\|_{\Psi(L),Q}.
\end{equation}
In the particular case $\Psi(t)=t^r$, $1\le r<\infty$ then $M_{\Psi(L)}$ coincides with the maximal operator defined in \eqref{DefMq},
$$
M_{\Psi(L)}f(x)=M_{L^r}f(x)=\sup_{Q\ni x} \Big(\frac{1}{|Q|}\int_Q|f|^r\,dt\Big)^{1/r},
$$
so in this case, we will still use an specific notation, denoting $M_{L^r}$ as $M_r$. Obviously, for $r=1$ we recover the usual Hardy--Littlewood maximal function.

Some particular interesting cases are the
Young functions
$$
\Phi(t)= t\,(1+\log^+t)^{\delta} \quad {\rm and}\quad \Psi(t)=e^t-1,
$$
defining the classical Zygmund spaces $L(\log L)$, in the case $\delta=1$,  and $\exp L$
respectively. The corresponding averages will be denoted by
$$
\|\cdot\|_{\Phi,Q} = \|\cdot\|_{L(\log L)^{\delta},Q}\quad {\rm and}
\quad\|\cdot\|_{\Psi,Q} = \|\cdot\|_{\exp L,Q}.
$$

It follows from the definition of norm that given two Young functions $\Psi(t)$ and $\Phi(t)$ such that for some $\kappa, c>0$ $\Psi(t)\leq \kappa\,\Phi(t)$ for every $t>c$, then
$\|f\|_{\Psi(L),Q} \leq (\Psi(c)+\kappa) \|f\|_{\Phi(L),Q}$ and hence $M_{\Psi(L)}f(x)\leq (\Psi(c)+\kappa)\, M_{\Phi(L)}f(x)$  for each $x$. In particular and since $(1+\log^{+}t ) \leq \frac{1}{\delta}t^{\delta}$, $\delta>0$, we have %
\begin{equation}
\label{eq:desig}
Mf(x)\leq M_{L(\log L)}f(x)\leq r'\,M_rf(x) \quad r>1,\quad x\in \R^n.
\end{equation}
This example fulfills our purpose in \eqref{maximalScale}.

Associated to each Young function $\Phi$,  one can define a
complementary function
\begin{equation}\label{complementaria}
\bar \Phi(s)=\sup_{t>0}\{st-\Phi(t)\}.
\end{equation}
Such $\bar \Phi$ is also a Young function and the
$\bar\Phi$-averages it defines are related to the
$L_{\Phi}$-averages  via a {\it generalized H\"older's
inequality}, namely,
\begin{equation*}
\frac1{|Q|}\,\int_{Q}|f(x)\,g(x)|\,dx \le
2\,\|f\|_{\Phi,Q}\,\|g\|_{\bar\Phi,Q}.
\end{equation*}
As a particular case the following holds
\[
\frac{1}{|Q|}\int_Q |b-b_Q||f|\leq c \|b-b_Q\|_{\exp L,Q}\|f\|_{L\log L,Q}.
\]
If we also know that $b\in \BMO$ then by John--Nirenberg Theorem (see for instance \cite[Theorem 6.11]{Duolibro} or \cite[Theorem 7.1.6]{G}) we have that
\begin{equation}\label{HGenBMOLlogL}
\frac{1}{|Q|}\int_Q |b-b_Q||f|\leq c_n \|b\|_{\BMO}\|f\|_{L\log L,Q}.
\end{equation}

For a more detailed account about the notions presented in this subsection we refer the reader to \cite{RR}.

\section{Lemmata}
\label{Lemmata}

\subsection{Unweighted $L^p$ estimates with good decay for $\widetilde{T}_j^N$ and $[b,\widetilde{T}_j^N]$}

For $\widetilde{T}^N_j$ a $L^p$ estimate with good decay was obtained in \cite[Lemma 3.14]{HRT}.
\begin{lemma}[\cite{HRT}]
\label{lem:Lpunweighted}
Let $1<p<\infty$. Then we have
\begin{equation*}
  \|\widetilde{T}_j^Nf\|_{L^p} \le c_{n,p}\|\Omega\|_{L^\infty}2^{-\alpha_p N(j-1)}\big(1+N(j)\big) \|f\|_{L^p},
\end{equation*}
for some numerical $0<\alpha_p<1$ independent of $T_{\Omega}$ and the function $N(\cdot)$.
\end{lemma}

For $[b,\widetilde{T}_j^N]$, we will obtain an analogous estimate as in Lemma \ref{lem:Lpunweighted} by using the \textit{method of conjugation}. Given an operator $T$ and $z$ a complex number, this technique consists of writing $[b,T]$ as a complex integral operator, by means of the Cauchy integral theorem, as
$$
[b,T]f=\frac{d}{dz}e^{zb}T
(fe^{-zb})\Big|_{z=0}=\frac{1}{2\pi i}\int_{|z|=\varepsilon}\frac{e^{zb}T(fe^{-zb})}{z^{2}}dz, \quad \varepsilon>0.
$$
Here, $z\to e^{zb}T(fe^{-zb})$ is called the \textit{conjugation} of $T$ by $e^{zb}$. Then, by applying Minkowski inequality, we get
$$
\|[b,T]f\|_{L^p}\leq\frac{1}{2\pi \varepsilon} \sup_{|z|=\varepsilon}  \| e^{zb}T(fe^{-zb})\|_{L^p}.
$$

\begin{lemma}\label{LemmaUnweighted} Let $1<p<\infty$.
Then we have
\[
\|[b,\widetilde{T}_{j}^{N}]f\| _{L^{p}}\leq c_{n,p}\|b\|_{\BMO}\|\Omega\|_{L^{\infty}}2^{-\alpha_{p,n}N(j-1)}\big(1+N(j)\big)\|f\|_{L^{p}}.
\]
\end{lemma}
\begin{proof}
We may assume that $\|b\|_{\BMO}=1$. We use the conjugation method so, for any $\varepsilon>0$, we write
\[
 [b,\widetilde{T}_{j}^{N}]f=\frac{1}{2\pi i}\int_{|z|=\varepsilon}\frac{e^{zb}\widetilde{T}_{j}^{N}(fe^{-zb})}{z^{2}}dz.
\]
Therefore, for any $\varepsilon$,
\[
\|[b,\widetilde{T}_{j}^{N}]f\|_{L^{p}}\leq\frac{1}{2\pi \varepsilon}  \sup_{|z|=\varepsilon}  \| e^{zb}\widetilde{T}_{j}^{N}(fe^{-zb})\|_{L^{p}}=\frac{1}{2\pi \varepsilon} \sup_{|z|=\varepsilon}  \| \widetilde{T}_{j}^{N}(fe^{-zb})\|_{L^{p}(e^{p\Re zb})}.
\]

Let us fix $\varepsilon>0$ to be chosen later and then take $z$ such that $|z|=\varepsilon$. Let $w_{z}=e^{p\Re zb}$, $v_{z}=e^{zb}$, and
$W_{z}=\left[w_{z}\right] _{A_{p}}^{\max\left\{1,\frac{1}{p-1}\right\}}$.  Our goal is to derive estimates uniformly in $z$.
We observe that, as in the proof of \cite[Theorem 1.4]{HRT} (see Subsection 3.4 therein), in particular from the use of the RHI in Theorem \ref{thm:SharpRHI} and an application of the interpolation theorem with change of measures (see \cite{BL} or \cite[Theorem 2.11]{Stein-Weiss}), the following estimate can be derived
\begin{align*}
\|
\widetilde{T}_{j}^{N}(fv_{z}^{-1})\|_{L^{p}(w_{z})}&\leq c_{n,p}\|\Omega\|_{L^{\infty}}\big(1+N(j)\big)2^{-\alpha_{p,n}\frac{N(j-1)}{c_n  W_{z}}}
W_{z}\|fv_{z}^{-1}\|_{L^{p}(w_{z})}\\
&\leq c_{n,p}\|\Omega\|_{L^{\infty}}\big(1+N(j)\big)2^{-\alpha_{p,n}\frac{N(j-1)}{c_nW_{z}}}
W_{z}\|f\|_{L^{p}}.
\end{align*}
It follows essentially from the John-Nirenberg lemma (see \cite[Lemma 2.2]{CPP}) the following general fact: if $p\in(1,\infty)$, there are two dimensional constants $\alpha_n,\beta_n$ such that \footnote{The observation in \eqref{eq:obser} can be found after the proof of \cite[Lemma 2.2, p. 1167]{CPP}. However, there is a misprint therein: $\min\big\{1,\frac{1}{p-1}\big\} $ should be $\min\{1,p-1\}$.}
\begin{equation}
\label{eq:obser}
s \in \R, \quad  |s| \leq \frac{\alpha_n}{\|b\|_{\BMO}}  \min\{1,p-1\}     \Rightarrow e^{s\,b}\in A_p \quad
\mbox{and} \quad [e^{s\,b}]_{A_p}\leq \beta^p_{n} 
\end{equation}
Taking into account the latter and that $\|b\|_{\BMO}=1$, the choice $\varepsilon= \alpha_n \min\{\frac1p,\frac{1}{p'}\}$ yields $W_{z}\leq c_{n,p}$ whenever $|z|\leq \varepsilon$, concluding the proof of the lemma.

\end{proof}

\subsection{Two-weight estimate for $\widetilde{T}_j^N$}

In this subsection we present the following two-weight estimate.
\begin{lemma}
\label{lem:Lpweighted}
Let $1<p<\infty$ and $r>1$. Then, for all $w\ge0$, we have
$$
\|\widetilde{T}_j^N f\|_{L^p(w)}\le  c_{n}\|\Omega\|_{L^\infty}\big(1+N(j)\big)pp'(r')^{\frac{1}{p'}} \|f\|_{L^p(M_rw)}.
$$
\end{lemma}

Lemma \ref{lem:Lpweighted} follows from several known facts. One of them is the quantitative pointwise domination (Theorem \ref{thm:domination}). The other one is a particular case of \cite[Lemma 4.3]{HP2} which in turn it is based on the original arguments from \cite{LOP1,LOP2}.

\begin{lemma}[\cite{HP2}]
\label{lem:sparseHP}
Let $w$ be any weight, $1<p<\infty$, $r>1$ and $\D$ a dyadic lattice. Then for any sparse family  $\Ss \subseteq \D$ and $g\ge0$, we have that
$$
\|\mathcal{A}_{ \Ss}g\|_{L^{p'}((M_rw)^{1-p'})}\le C p'\|Mg\|_{L^{p'}((M_rw)^{1-p'})}.
$$
\end{lemma}

The proof of Lemma \ref{lem:Lpweighted} follows the scheme in \cite[pp. 617--619]{HP2}, combined with Theorem \ref{thm:domination}, Lemma \ref{lem:modulus}, and Lemma \ref{lem:sparseHP}. We skip the details.

\subsection{Two-weight inequality for $[b,\widetilde{T}_j^N]$}
In this subsection we establish a two-weight estimate for $[b,\tilde{T}_j^N]$, which is very similar to the one obtained in Lemma \ref{lem:Lpweighted}.
\begin{lemma}
\label{LemmaMr}Let $1<p<\infty$ and $r>1$. Then, for all $w\geq0$,
we have
\[
\|[b,\tilde{T_{j}^{N}}]f\|_{L^{p}(w)}\leq c_{n}\|b\|_{\BMO}\|\Omega\|_{L^{\infty}}\big(1+N(j)\big)(p'p)^{2}(r')^{1+\frac{1}{p'}}
\|f\|_{L^{p}(M_{r}w)}.
\]
\end{lemma}
The proof of Lemma \ref{LemmaMr} is similar to the proof of Lemma \ref{lem:Lpweighted} but we include it for completeness. We need a generalization of \cite[Lemma 4.1]{HP2} whose proof will be based on  the following dyadic version of the Carleson embedding theorem (see for instance \cite[Theorem 4.5]{HP1}):
\begin{theorem} [\cite{HP1}]
\label{thm:Carleson}
Let $\mathscr{D}$ be a dyadic lattice and let $\{a_Q\}_{Q \in \mathscr{D}}$ be a sequence of nonnegative numbers  satisfying the \textit{Carleson condition}
$$
\sum_{Q\subseteq R} a_Q\le A w(R), \quad R\in \mathscr{D},
$$
for some constant $A>0$.
Then, for all $p\in (1,\infty)$ and $f\in L^p(w)$,
$$
\bigg(\sum_{Q\in \mathscr{D}}a_Q\Big(\frac{1}{w(Q)}\int_Q
f(x)w(x)dx\Big)^p\bigg)^{1/p}\le A^{1/p}\cdot p'\cdot \|f\|_{L^p(w)}.
$$
\end{theorem}
\begin{lemma}\label{4.1HP2Gen}
Let $w\in A_\infty$. Let $\mathscr{D}$ be a dyadic lattice and $\mathscr{S}\subset\mathscr{D}$ be an $\eta$-sparse family. Let $\Psi$ be a Young function. Given a measurable function $f$  on $\mathbb{R}^n$ define
\[
\mathcal{B}_\mathscr{S}f(x):=\sum_{Q\in\mathscr{S}}\|f\|_{\Psi(L),Q}\chi_Q(x).
\] 
Then we have 
\[
\|\mathcal{B}_\mathscr{S}f\|_{L^1(w)}\leq \frac{4}{\eta}[w]_{A_\infty}\|M_{\Psi(L)}f\|_{L^1(w)}.
\]
\end{lemma}
\begin{proof} First, we see that
\begin{align*}
\|\mathcal{B}_\mathscr{S}f\|_{L^1(w)}&=\sum_{Q\in\mathscr{S}}\|f\|_{\Psi(L),Q}w(Q)
\leq\sum_{Q\in\mathscr{S}}\big(\inf_{Q\in z}M_{\Psi(L)}f(z)\big)w(Q)\\
&\leq\sum_{Q\in\mathscr{S}}\Big(\frac{1}{w(Q)}\int_Q
\big(M_{\Psi(L)}f(x)\big)^\frac{1}{2}w(x)dx
\Big)^2w(Q).
\end{align*}
Applying Carleson embedding theorem (Theorem \ref{thm:Carleson}) with $g=(M_{\Psi(L)}f)^\frac{1}{2}$ we have that
\[\sum_{Q\in\mathscr{S}}\Big(\frac{1}{w(Q)}\int_Q g w(x)d(x)\Big)^2w(Q)\leq 4A\|g\|^2_{L^2(w)}=4A\|M_{\Psi(L)}f\|_{L^1(w)}\]
provided we can show that the Carleson condition
\[\sum_{\substack{Q\subseteq R\\
R\in\mathscr{S}}}w(Q)\leq A w(R)\]holds. We observe that
\begin{align*}
\sum_{\substack{Q\subseteq R\\
R\in\mathscr{S}}}w(Q)&\leq \sum_{\substack{Q\subseteq R\\
R\in\mathscr{S}}}\frac{w(Q)}{|Q|}|Q|
\leq \sum_{\substack{Q\subseteq R\\
R\in\mathscr{S}}} \inf_{z\in Q}M(\chi_Rw)(z)\frac{1}{\eta}|E_Q|\\
& \leq \frac{1}{\eta} \int_R M(\chi_Rw)(z)dz \leq\frac{1}{\eta}[w]_{A_\infty}w(R).
\end{align*}
Then we have that the Carleson condition holds with $A=\frac{1}{\eta}[w]_{A_\infty}$. This ends the proof of the lemma.
\end{proof}

Now we are in position to prove Lemma \ref{LemmaMr}.
\begin{proof}[Proof of Lemma \ref{LemmaMr}]  We follow some of the key ideas from \cite{LOP1,LOP2} (see also \cite{HP2}).
By duality, it suffices to prove that
\[
\bigg\| \frac{[b,\widetilde{T}_{j}^{N}]f}{M_{r}w}\bigg\|_{L^{p'}(M_{r}w)}\leq c_{n}\big(1+N(j)\big)\|\Omega\|_{L^{\infty}}\|b\|_{\BMO}(p'p)^{2}(r')^{1+\frac{1}{p'}}
\Big\| \frac{f}{w}\Big\|_{L^{p'}(w)}.
\]
We calculate the norm by duality, so we have that
\[
\bigg\| \frac{[b,\widetilde{T}_{j}^{N}]f}{M_{r}w}\bigg\|_{L^{p'}(M_{r}w)}=\sup_{\|h\|_{L^{p}(M_{r}w)}=1}
 \int_{\mathbb{R}^{n}} \big| [b,\widetilde{T}_{j}^{N}]f(x)\big| h(x)dx. 
\]
We need a version of the Rubio de Francia algorithm suited for this situation (see \cite[Chapter IV.5]{GCRdF} and \cite{CUMP} for a detailed account on Rubio de Francia algorithm and  several applications).
Consider the operator
$$S(f)=\frac{M\big(f(M_{r}w)^{\frac{1}{p}}\big)}{(M_{r}w)^{\frac{1}{p}}}$$
and observe that $S$ is bounded on $L^{p}(M_{r}w)$ with norm bounded by a dimensional multiple of $p'$.
We define
\[
R(h)=\sum_{k=0}^{\infty}\frac{1}{2^{k}}\frac{S^{k}h}{\|S\|_{L^{p}(M_{r}w)}^{k}}.
\]
This operator has the following properties:
\begin{enumerate}
\item[(a)] $0\leq h\leq R(h)$,
\item[(b)] $\|Rh\|_{L^{p}(M_{r}w)}\leq2\|h\|_{L^{p}(M_{r}w)}$,
\item[(c)] $R(h)(M_{r}w)^{\frac{1}{p}}\in A_{1}$ with $[R(h)(M_{r}w)^{\frac{1}{p}}]_{A_{1}}\leq cp'$.
\end{enumerate}
We also observe that $[Rh]_{A_{\infty}}\leq[Rh]_{A_{3}}\leq c_{n}p'$.
Now
\[
\int_{\mathbb{R}^{n}} \big|  [b,\widetilde{T}_{j}^{N}]f(x)\big|h(x)dx \leq \int_{\mathbb{R}^{n}}\big|[b,\widetilde{T}_{j}^{N}]f(x)\big|Rh(x)dx.
\]
Let $C_{\widetilde{T}_{j}^{N}}$ denote the constant defined in \eqref{eq:CT}, related to the operator $\widetilde{T}_{j}^{N}$. Using Theorem~\ref{thm:ConmDomination} we get
\[
\int_{\mathbb{R}^{n}}\Big|[b,\widetilde{T}_{j}^{N}]f(x)\Big|Rh(x)dx\leq c_{n}C_{\widetilde{T}_{j}^{N}}\int_{\mathbb{R}^{n}}\sum_{j=1}^{3^{n}}
\big(\mathcal{T}_{\mathscr{S}_{j},b}|f|(x)+\mathcal{T}_{\mathscr{S}_{j},b}^{*}|f|(x)\big)Rh(x)dx
\]
and it suffices to obtain estimates for
\[
I:=\int_{\mathbb{R}^{n}}\mathcal{T}_{\mathscr{S}_{j},b}|f|(x)Rh(x)dx\qquad\text{and}\qquad II:=\int_{\mathbb{R}^{n}}\mathcal{T}_{\mathscr{S}_{j},b}^{*}|f|(x)Rh(x)dx.
\]
First we focus on $I$. Choosing $s=1+\frac{1}{c_{n}[Rh]_{A_{\infty}}}$,
by \eqref{HGenBMOLlogL}, \eqref{eq:desig}, RHI, and the property~(c) above, we have
\begin{align*}I  =\int_{\mathbb{R}^{n}}\mathcal{T}_{\mathscr{S}_{j},b}|f|(x)Rh(x)dx&\leq\sum_{Q\in\mathscr{S}_{j}}\int_{Q}|b(x)-b_{Q}|Rh(x)dx\frac{1}{|Q|}\int_{Q}|f|dy\\
 & \leq2\|b\|_{BMO}\sum_{Q\in\mathscr{S}_{j}}\|Rh\|_{L\log L,Q}\int_{Q}|f|dy\\
 &\leq 2s'\|b\|_{BMO}\sum_{Q\in\mathscr{S}_{j}}\left(\frac{1}{|Q|}\int_{Q}Rh^{s}\right)^{\frac{1}{s}}\int_{Q}|f|dy\\
 & \leq c_n[Rh]_{A_{\infty}}\|b\|_{\BMO}\sum_{Q\in\mathscr{S}_{j}}Rh(Q)\frac{1}{|Q|}\int_{Q}|f|dy\\
 &\leq c_{n}p'\|b\|_{\BMO}\sum_{Q\in\mathscr{S}_{j}}Rh(Q)\frac{1}{|Q|}\int_{Q}|f|dy.
\end{align*}
We apply now Lemma \ref{4.1HP2Gen} with $\Psi(t)=t$, thus
\[
\sum_{Q\in\mathscr{S}_{j}}Rh(Q)\frac{1}{|Q|}\int_{Q}|f|dy\leq8[Rh]_{A_{\infty}}\|Mf\|_{L^{1}(Rh)}\leq c_np'\|Mf\|_{L^{1}(Rh)}.
\]
From here
\[
\|Mf\|_{L^{1}(Rh)}\leq\Big(\int_{\mathbb{R}^{n}}|Mf|^{p'}\big(M_{r}w\big)^{1-p'}\Big)^{\frac{1}{p'}}\Big(\int_{\mathbb{R}^{n}}(Rh)^{p}M_{r}w\Big)^{\frac{1}{p}}\leq2\Big\| \frac{Mf}{M_{r}w}\Big\| _{L^{p'}(M_{r}w)}.
\]
By \cite[Lemma 3.4]{LOP1} (see also \cite[Lemma 2.9]{OC})
\[
\Big\|\frac{Mf}{M_{r}w}\Big\|_{L^{p'}(M_{r}w)}\leq c p(r')^{\frac{1}{p'}}\Big\|\frac{f}{w}\Big\|_{L^{p'}(w)}.
\]
Summarizing,
\[
I\leq c_{n}\|b\|_{\BMO}p(p')^{2}(r')^{\frac{1}{p'}}\left\Vert \frac{f}{w}\right\Vert _{L^{p'}(w)}.
\]

Now we focus on $II$. A direct application of \eqref{HGenBMOLlogL} yields
\[
II\leq\sum_{Q\in\mathscr{S}_{j}}\Big(\frac{1}{|Q|}\int_{Q}|b(y)-b_{Q}|f(y)dy\Big)Rh(Q)\leq c_n\|b\|_{\BMO}\sum_{Q\in\mathscr{S}_{j}}\|f\|_{L\log L,Q}Rh(Q)
\]
Then, we use Lemma \ref{4.1HP2Gen} with $\Psi(t)=t\log(e+t)$, leading to the following estimate
\[
\sum_{Q\in\mathscr{S}_{j}}\|f\|_{L\log L,Q}Rh(Q)\leq8[Rh]_{A_{\infty}}\|M_{L\log L}f\|_{L^{1}(Rh)}.
\]
Proceeding as in the estimate of $I$,
\[
\|M_{L\log L}f\|_{L^{1}(Rh)}\leq2\Big\|\frac{M_{L\log L}f}{M_{r}w}\Big\| _{L^{p'}(M_{r}w)}.
\]
Now \cite[Proposition 3.2]{OC} gives
\[
\Big\|\frac{M_{L\log L}f}{M_{r}w}\Big\|_{L^{p'}(M_{r}w)}\leq c_{n}p^{2}(r')^{1+\frac{1}{p'}}\Big\|\frac{f}{w}\Big\|_{L^{p'}(w)}.
\]
We combine all the estimates and we have that
\[
II\leq c_{n}\|b\|_{\BMO}p'p^{2}(r')^{1+\frac{1}{p'}}\Big\| \frac{f}{w}\Big\|_{L^{p'}(w)}.
\]
Finally, collecting the estimates we have obtained for $I$ and $II$
\[
\Big\|\frac{[b,\widetilde{T}_{j}^{N}]f}{M_{r}w}\Big\|_{L^{p'}(M_{r}w)}\leq c_{n}C_{\widetilde{T}_{j}^{N}}\|b\|_{\BMO}(p'p)^{2}(r')^{1+\frac{1}{p'}}\Big\|\frac{f}{w}\Big\|_{L^{p'}(w)}.
\]
Since $C_{\widetilde{T}_{j}^{N}}\leq c_{n}\|\Omega\|_{L^{\infty}}(1+N(j))$ (recall the definition in \eqref{eq:CT}), we arrive at the desired bound. The proof of the lemma is complete.
\end{proof}

\subsection{A basic summation result}

In this subsection we prove the following key result that was used in \cite{HRT} and which will be used in the proof of Theorem \ref{thm:main1}.
\begin{lemma}\label{Summation}
Let $N(j)=2^j$ for $j> 0$ and $N(-1)=N(0)=0$. Let $\alpha>0$ and $0<\theta<1$. Then  there exists some $c>0$ depending on $\alpha$ such that
$$
\sum_{j=0}^\infty  (1+N(j))2^{-\alpha N(j-1) \theta} \leq \frac{c}{\theta}.
$$
\end{lemma}
\begin{proof}
Let us take $\{N(j)\}_{j\ge0}$ so that $N(j)=2^j$ for $j\geq 1$ and $N(0)=0$, and let us establish $N(-1)=0$. Observe also that $e^x\geq \frac12 x^2$ and hence
 \begin{equation}
\label{eq:obvio}
e^{-x}\leq 2 x^{-2}.
\end{equation}
We split the sum into two parts
$$
\sum_{j=0}^\infty (1+N(j))2^{-\alpha N(j-1)\theta}=\Big(\sum_{j:2^j\leq\theta^{-1}}+\sum_{j:2^j\geq\theta^{-1}}\Big)  (1+N(j))2^{-\alpha N(j-1)\theta}=:S_1+S_2.
$$
 For $S_1$, we have
  \begin{align*}
 S_1& \leq c\sum_{j:2^j\leq\theta^{-1}}2^j2^{-C2^j\theta}\le c\sum_{j:2^j\leq\theta^{-1}}2^j\le c\theta^{-1}.
\end{align*}
On the other hand, for $S_2$, by using \eqref{eq:obvio},
\begin{align*}
 S_2\le c\sum_{j:2^j\geq\theta^{-1}}2^j\Big(\frac{1}{2^j\theta}\Big)^2
 \leq \frac{c}{\theta},
\end{align*}
and the conclusion follows.
\end{proof}

\section{Proofs of the main results}\label{PruebasThmCor}
\subsection{Proof of Theorem \ref{thm:main1}}
The preceding section contained unweighted estimates for each of the pieces $\widetilde{T}_j^N$ (and also for the commutators of these pieces) which have a very good decay of exponential type whereas their corresponding  weighted estimates are of polynomial growth nature. Following the strategy in \cite[Subsection 3.4]{HRT}, the main idea to prove Theorem \ref{thm:main1} is then to somehow combine both estimates to obtain the desired result. That combination can be performed by using  interpolation with change of measures.

Let us establish the result for $T_\Omega$  first. In view of Lemma \ref{lem:Lpunweighted} and Lemma \ref{lem:Lpweighted}, by applying the interpolation theorem with change of measures (\cite[p.115, Theorem 5.4.1]{BL} or \cite[Theorem 2.11]{Stein-Weiss}) we have that, for $0<\theta<1$,
$$
\|\widetilde{T}_j^N f\|_{L^p(w^{\theta})}
\le  c_{n,p}\|\Omega\|_{L^\infty}(p')^{\theta}2^{-\alpha_p N(j-1)(1-\theta)}\big(1+N(j)\big)(r')^{\frac{\theta}{p'}} \|f\|_{L^p(M_{r/\theta}(w^{\theta}))}.
$$
Thus
\begin{align*}
  \|T_\Omega f&\|_{L^p(w^{\theta})}\leq\sum_{j=0}^\infty\|\widetilde{T}_j^N f\|_{L^p(w^{\theta})} \\
  &\leq C_{n,p}\|\Omega\|_{L^\infty}(r')^{\frac{\theta}{p'}} \|f\|_{L^p(M_{r/\theta}(w^{\theta}))}\sum_{j=0}^\infty (1+N(j))2^{-\alpha_{p,n}N(j-1)(1-\theta)}.
\end{align*}
Now a direct application of Lemma \ref{Summation} yields
$$
 \|T_\Omega f\|_{L^p(w^{\theta})}\le C_{n,p}\|\Omega\|_{L^\infty}\frac{1}{1-\theta}(r')^{\frac{\theta}{p'}}\|f\|_{L^p(M_{r/\theta}(w^{\theta}))},
$$
or, calling $u:=w^{\theta}$,
$$
 \|T_\Omega f\|_{L^p(u)}\le C_{n,p}\|\Omega\|_{L^\infty}\frac{1}{1-\theta}(r')^{\frac{\theta}{p'}}\|f\|_{L^p(M_{r/\theta}(u))}.
$$
This yields the proof of \eqref{eq:desigr}.

Now we consider the operator $[b,T_\Omega]$. By using Lemmas \ref{LemmaUnweighted} and \ref{LemmaMr}, again by the interpolation theorem
with change of measures we obtain, for $0<\theta<1$,
\begin{align*}
\big\| [b,\widetilde{T}_{j}^{N}]f\big\|_{L^{p}(w^{\theta})}
\leq & c_{n,p}\|b\|_{\BMO}\|\Omega\|_{L^{\infty}}(r')^{\theta\big(1+\frac{1}{p'}\big)}\|f\|_{L^{p}(M_{r/\theta}(w^{\theta}))}\\
&\qquad \times (1+N(j))2^{-\alpha_{p,n}N(j-1)(1-\theta)}.
\end{align*}
Consequently,
\begin{align*}
\big\|[b,T_{\Omega}]f\big\|_{L^{p}(w^{\theta})}
\leq & \sum_{j=0}^{\infty}\big\|[b,\widetilde{T}_{j}^{N}]f\big\|_{L^{p}(w^{\theta})}\\
\leq & c_{n,p}\|b\|_{\BMO}\|\Omega\|_{L^{\infty}}(r')^{\theta\big(1+\frac{1}{p'}\big)}\|f\|_{L^{p}(M_{r/\theta}(w^{\theta}))}
\\
&\qquad \times  \sum_{j=0}^{\infty}(1+N(j))2^{-\alpha_{p,n}N(j-1)(1-\theta)}.
\end{align*}
Using again Lemma \ref{Summation}, we have that
\[
\sum_{j=0}^{\infty}(1+N(j))2^{-\alpha_{p,n}N(j-1)(1-\theta)}\leq C_{n,p}(1-\theta)^{-1}.
\]
In conclusion, we get
\[
\big\|[b,T_{\Omega}]f\big\|_{L^{p}(w^{\theta})}\leq c_{n,p}\|b\|_{\BMO}\|\Omega\|_{L^{\infty}}\frac{1}{1-\theta}(r')^{\theta\big(1+\frac{1}{p'}\big)}\|f\|_{L^{p}(M_{r/\theta}(w^{\theta}))}
\]
and calling $u:=w^{\theta},$
\[
\big\|[b,T_{\Omega}]f\big\|_{L^{p}(u)}\leq c_{n,p}\|b\|_{\BMO}\|\Omega\|_{L^{\infty}}\frac{1}{1-\theta}(r')^{\theta\big(1+\frac{1}{p'}\big)}\|f\|_{L^{p}(M_{r/\theta}(u))}
\]
%
we obtain \eqref{eq:desigrComm}, concluding the proof of Theorem~\ref{thm:main1}.

\subsection{Proof of Corollary \ref{cor1}}
We focus on the result concerning $T_\Omega$. We omit the proof for $[b,T_\Omega]$ since it follows analogously with minor modifications.

Fix  $p,s \in (1,\infty)$. By Theorem \ref{thm:main1}  with $r=\theta s$ we have that, for $\theta \in (\frac1s,1)$,
$$
 \|T_\Omega f\|_{L^p(w)}\le C_{n,p}\|\Omega\|_{L^\infty}\frac{1}{1-\theta}\big((\theta s)'\big)^{\frac{\theta}{p'}}\|f\|_{L^p(M_{s}(w))}.
$$
Now we choose $\theta=\frac{(2s')'}{s}$. We observe that $\theta \in (\frac1s,1)$ and also
that $\frac{1}{1-\theta} =2s'-1$. Hence, we conclude that for any $s>1$,
$$
 \|T_\Omega f\|_{L^p(w)}\le C_{n,p}\|\Omega\|_{L^\infty} (s')^{1+\frac{1}{p'}}\|f\|_{L^p(M_{s}(w))}  \qquad w\geq 0.
$$
Now, if $w\in A_{\infty}$, by Theorem \ref{thm:SharpRHI} there exists a dimensional constant $\tau_{n}$ and we can choose $s_{w}= 1+\frac{1}{\tau_{n}[w]_{A_{\infty}}}$ such that
\begin{equation*}
  \Big(\frac{1}{|Q|}\int_{Q}w^{s_{w}}\Big)^{\frac{1}{s_{w}}}\leq\frac{2}{|Q|}\int_{Q}w.
\end{equation*}
 Then,
\begin{align*}
\|T_\Omega f\|_{L^p(w)} &\leq C_{n,p}\|\Omega\|_{L^\infty} (s_{w}')^{1+\frac{1}{p'}}\|f\|_{L^p(M_{s_{w}}(w))}\\
&\leq C_{n,p}\,\|\Omega\|_{L^\infty}   [w]_{A_{\infty}}^{1+1/p'}  \,\|f\|_{L^p(M(w))},
\end{align*}
and hence
$$
\|T_\Omega\|_{L^p(w)} \leq C_{n,p}\,\|\Omega\|_{L^\infty}   [w]_{A_{\infty}}^{1+ \frac{1}{p'} }  \,[w]_{A_1}^{\frac{1}{p}}.
$$
This ends the proof of Corollary \ref{cor1}.

\section*{Acknowledgements}
The second author would like to thank the hospitality and the support of the Department of Mathematics   of the University of the Basque Country (UPV/EHU) for allowing him to use their facilities during the development of this research project.

\section*{References}

\end{document}